\documentclass{article}

\usepackage{yfonts}
\usepackage{amsthm,amssymb,amsfonts,mathrsfs}
\usepackage{amsmath}

\newcommand{\RN}[1]{%
	\textup{\uppercase\expandafter{\romannumeral#1}}%
}

\newtheorem{Theorem}{Theorem}[section]
\newtheorem{Proposition}{Proposition}[section]
\newtheorem{Lemma}{Lemma}[section]
\newtheorem{Corollary}{Corollary}[section]
\theoremstyle{definition}

\newtheorem{Definition}{Definition}[section]
\newtheorem{Remark}{Remark}

\newtheorem{Assumptions}{Hypothesis}[section]
\def\ve{\varepsilon}
\def\N{{\mathbb{N}}}

\def\R{{\mathbb{R}}}

\def\cA{{\mathcal{A}}}

\def\cU{{\mathcal{U}}}

\def\t\cU{{\widetilde{{\mathcal{U}}}}}

\newcommand\norm[1]{\left\lVert#1\right\rVert}

\def\ds{\displaystyle}

\title {Boundary controllability for a degenerate beam equation}
\author{{{\sc Alessandro Camasta}\thanks{The author is a member of the  {\it Gruppo Nazionale per l'Analisi Ma\-te\-matica, la Probabilit\`a e le loro Applicazioni (GNAMPA)} of the Istituto Nazionale di Alta Matematica (INdAM), a member of {\it UMI ``Modellistica Socio-Epidemiologica (MSE)''} and he is partially supported by PRIN 2017-2019 {\it Qualitative and quantitative aspects of nonlinear PDEs.} He is also supported by the project {\it Mathematical models for interacting dynamics on networks (MAT-DYN-NET)
CA18232.}}}\\
	Department of Mathematics\\ University of Bari Aldo Moro\\
	Via
	E. Orabona 4\\ 70125 Bari - Italy\\ e-mail: alessandro.camasta@uniba.it\\
	{\sc Genni Fragnelli}\thanks{The author is a member of the  {\it Gruppo Nazionale per l'Analisi Ma\-te\-matica, la Probabilit\`a e le loro Applicazioni (GNAMPA)} of the Istituto Nazionale di Alta Matematica (INdAM), a member of {\it UMI ``Modellistica Socio-Epidemiologica (MSE)''} and is supported by FFABR {\it Fondo per il finanziamento delle attivit\`a base di ricerca} 2017, by  PRIN 2017-2019 {\it Qualitative and quantitative aspects of nonlinear PDEs} and by the
HORIZON$_-$EU$_-$DM737 project 2022 {\it COntrollability of PDEs in the Applied Sciences (COPS)} at Tuscia University. She is also supported by the project {\it Mathematical models for interacting dynamics on networks (MAT-DYN-NET)
CA18232.}}\\
	Department of Ecology and Biology\\ Tuscia University\\
	Largo dell'Universit\`a, 01100 Viterbo - Italy\\ e-mail: genni.fragnelli@unitus.it}

\date{}
\begin{document}
	
	\maketitle

	\vspace{0.3cm}
	
	\centerline{ {\it  }}

	\begin{abstract}
		The paper deals with the controllability of a degenerate  beam equation. In particular, we assume that the left end of the beam is fixed, while a suitable control $f$ acts on the right end of it. As a first step we prove the existence of a solution for the homogeneous problem, then we prove some estimates on its energy. Thanks to them we prove an observability inequality and, using the  notion of solution by transposition, we prove that the initial problem is null controllable.
	\end{abstract}
	

	\noindent 2000AMS Subject Classification: 
	35L10, 35L80, 93B05, 93B07, 93D15

\section{Introduction}
We consider  a boundary controllability problem for a system modelling the bending vibrations of a degenerate beam of length $L=1$. Denote by $u$ the deflection of the beam and assume that the left end of the beam is fixed, while a suitable shear force $f$ is exerted on the right end of the beam; thus the motion describing beam bending is given by the following problem
	\begin{equation}\label{(P)}
		\begin{cases}
			u_{tt}(t,x)+Au(t,x)=0, &(t,x)\in Q_T,\\
			u(t,0)=0,\,\,u_x(t,0)=0,&t\in (0,T),\\
			u(t,1)=0,\,\,u_x(t,1)=f(t), &t \in (0,T),\\
			u(0,x)=u_0(x),\,\,u_t(0,x)=u_1(x),&x\in(0,1),
		\end{cases}
	\end{equation}
where $Q_T:=(0,T) \times (0,1)$, $T>0$, and $Au:=au_{xxxx}$. 

An equation similar to the one of \eqref{(P)} can be found , for example, in models that describe the vibrations of a bridge. Indeed, a suspension bridge may be seen as a beam of given length $L$, with hinged ends and whose downward deflection is measured by a function $u(t,x)$ subject to three forces. These forces can be summarised as the stays holding the bridge up as nonlinear springs with spring constant $k$, the constant weight per unit length of the bridge $W$ pushing it down, and the external forcing term $f(t,x)$. This leads to the equation
\[
u_{tt} + \gamma u_{xxxx} =-k u^+ +W+f(t,x),
\]
where $\gamma$ is a physical constant depending on the beam, Young's modulus and the second moment of inertia. 
If $\gamma$ is a function that depends on the variable $x$ and the external function acts only on the boundary, then we have exactly the equation of \eqref{(P)} (see, e.g., \cite{CF_Neumann} for other applications of \eqref{(P)}).

The novelty of this paper is that $a:[0,1]\to\mathbb{R}$ is such that $a(0)=0$ and $a(x) > 0$ for all $x \in (0,1]$. If there exists a boundary function $f$ that drives the solution of \eqref{(P)} to $0$ at a given time $T>0$, in the sense that
\[
u(T,x)= u_t(T,x)=0
\]
for all $x \in (0,1)$, then the problem is said null controllable.

Boundary exact controllability on linear beam problems has been studied for many years by a lot of authors, see, for example,  \cite{bugariu}, \cite{krabs}-\cite{rao} and the references therein. For quasi-linear beams or nonlinear beams, we refer to \cite{yaoweiss} and \cite{cao}, \cite{cindea}, respectively. 

In all the previous papers the equation is always non degenerate. The first results on boundary controllability for degenerate problems can be found in \cite{ABCG}, \cite{BFM wave eq} and \cite{gueye}. In particular, in 
\cite{gueye} the author considers the equation in divergence form
\[
u_{tt}-(x^\alpha u_x)_x=0
\]
for $\alpha\in (0,1)$ and the control acts in the degeneracy point $x=0$. Later on, in \cite{ABCG} the authors consider the equation
\begin{equation}\label{div1}
u_{tt}-(a(x)u_x)_x=0
\end{equation}
where $a\sim x^K$, $ K >0$. In this  case the authors establish observability
inequalities when $K<2$; if $K \ge 2$, 
a negative result is given. We remark that  in \cite{gueye} the observability inequality, and hence null controllability, is obtained via spectral analysis, while in \cite{ABCG} via suitable energy estimates.
In \cite{BFM wave eq} the same problem of \eqref{div1} in non divergence case with a drift term is considered. Clearly, the presence of a drift term leads the authors to use different spaces with respect to the ones in \cite{ABCG} or in \cite{gueye} and gives rise to some new difficulties. However, thanks to some suitable assumptions on the drift term, the authors prove some estimates on the energy that are crucial to prove an observability inequality and hence null controllability for the initial problem.

As far as we know, this is the first paper where the boundary controllability for a {\it degenerate} beam equation is considered. For the function $a$, we consider two cases: the weakly degenerate case and the strongly degenerate one. More precisely, we have the following definitions:
\begin{Definition}
A function $a$ is {\it weakly degenerate at $0$}, $(WD)$ for short, if $a\in\mathcal{C}[0,1]\cap\mathcal{C}^1(0,1]$, $a(0)=0$, $a>0$ on $(0,1]$ and  if
	\begin{equation}\label{sup}
		K:=\sup_{x\in (0,1]}\frac{x|a'(x)|}{a(x)},
	\end{equation}
	then $K\in (0,1)$.
\end{Definition}
\begin{Definition}
A function $a$ is {\it strongly degenerate at $0$}, $(SD)$ for short,  if $a\in\mathcal{C}^1[0,1]$, $a(0)=0$, $a>0$ on $(0,1]$ and in \eqref{sup} we have $K\in [1,2)$.
\end{Definition}

Clearly,  the presence of the degenerate operator $Au:= au_{xxxx}$ leads us to  use different spaces with respect to the ones in \cite{ABCG}, \cite{BFM wave eq} or \cite{gueye} and in these new spaces we prove 
an estimate similar to the following one
\[
E_y(0) \le C \int_0^T y_{xx}^2(t,1)dt,
\]
where $y$ and $E_y$ are the solution and the energy of  the homogeneous adjoint problem associated to \eqref{(P)}, respectively, and $C$ is a strictly positive constant. Then, thanks to the introduction of the solution by transposition for \eqref{(P)}, we prove that \eqref{(P)} is null controllable.

The paper is organized in the following way: in Section \ref{sezione2} we consider the homogeneous problem associated to \eqref{(P)} and we prove that this problem is well posed in the sense of Theorem \ref{Theorem 2.6}; in Section \ref{Section 3} we consider the energy associated to it and we prove two estimates on the energy from below and from above. In Section \ref{Section 4}, thanks to these estimates and to the boundary observability (see Corollary \ref{cor 3.12}), we prove that the original problem has a unique solution by transposition and this solution is null controllable. The paper ends with the Appendix where we  give two proofs to make the article self-contained.

We underline that in the paper $C$ denotes universal positive constants which are allowed to vary from line to line.
\section{Well posedness for the problem with homogeneous Dirichlet boundary conditions}\label{sezione2}
In this section we study the well posedness of the following degenerate hyperbolic problem with Dirichlet boundary conditions
\begin{equation}\label{(P_1)}
	\begin{cases}
		y_{tt}(t,x)+Ay(t,x)=0, &(t,x)\in (0,+\infty) \times (0,1),\\
		y(t,0)=y(t,1)=0,&t\in (0,+\infty),\\
		y_x(t,0)=y_x(t,1)=0,&t\in (0,+\infty),\\
		y(0,x)=y^0_T(x),&x\in(0,1),\\
		y_t(0,x)=y^1_T(x),&x\in(0,1).
	\end{cases}
\end{equation}
We underline the fact that the choice of denoting initial data with $T$-dependence is connected to the approach for null controllability used in the next sections.

As in \cite{CF}, \cite{CF_Neumann} or \cite{CF_Wentzell} let us consider the following weighted Hilbert spaces:
\begin{equation*}
	L^2_{\frac{1}{a}}(0, 1):=\biggl \{u\in L^2(0, 1):\int_{0}^{1}\frac{u^2}{a}\,dx<+\infty \biggr \}
\end{equation*}
and
\begin{equation*}
	H^i_{\frac{1}{a}}(0, 1):= L^2_{\frac{1}{a}}(0, 1)\cap H^i_0(0, 1), \quad i=1,2,
\end{equation*}
with the related norms
\begin{equation*}
	\norm{u}^2_{L^2_{\frac{1}{a}}(0, 1)}:= \int_{0}^{1}\frac{u^2}{a}\,dx\quad\,\,\,\,\,\,\,\forall\,u\in L^2_{\frac{1}{a}}(0, 1)
\end{equation*}
and
\begin{equation*}
	\norm{u}_{H^i_{\frac{1}{a}}(0, 1)}^2:=\norm{u}^2_{L^2_{\frac{1}{a}}(0, 1)} + \sum_{k=1}^{i}\Vert u^{(k)}\Vert^2_{L^2(0, 1)} \quad\,\,\,\,\,\,\,\forall\,u\in H^i_{\frac{1}{a}}(0, 1),
\end{equation*}
$i=1,2$, respectively. We recall that $H^i_0(0, 1):= \{u\in H^i(0, 1): u^{(k)}(j)=0,\, j=0,1,\, k=0,...,i-1\}$, with $u^{(0)}=u$ and $i=1,2$.
Observe that for all $u \in H^i_{\frac{1}{a}}(0, 1)$, using the fact that $u^{(k)}(j)=0$ for all $k=0,...,i-1$ and $j=0,1$, it is easy to prove that $\|u\|_{H^i_{\frac{1}{a}}(0, 1)}^2$ is equivalent to the following one 
\[
\| u \|_i ^2:= \|u\|_{L^2_{\frac{1}{a}}(0, 1)}^2+ \|u^{(i)}\|_{L^2(0,1)}^2, \quad i=1,2.
\]
If $i=1$ the previous assertion is clearly true.

Moreover, under an additional assumption on $a$, one can prove that the previous norms are equivalent to the following one
\[
\|u\|_{i, \sim }:= \|u^{(i)}\|_{L^2(0,1)}, \quad i=1,2.
\]
Indeed, assume
\begin{Assumptions}\label{ipo1}
The function $a:[0,1]\to\mathbb{R}$ is continuous in $[0,1]$, $a(0)=0$, $a>0$ on $(0,1]$ and there exists $K\in (0,2)$ such that the function
	\begin{equation}\label{crescente}
		x\mapsto\frac{x^K}{a(x)}
	\end{equation}
is non decreasing near $x=0$.
\end{Assumptions}
Observe that if $a$ is weakly or strongly degenerate, then (\ref{sup}) implies that the function
\begin{equation*}
	x\mapsto\frac{x^\gamma}{a(x)}
\end{equation*}
is non decreasing in $(0,1]$ for all $\gamma\ge K$; in particular, \eqref{crescente} holds globally. Moreover,
\begin{equation}\label{ipotesi limite}
	\lim_{x\to 0}\frac{x^\gamma}{a(x)}=0
\end{equation}
for all $\gamma >K$.  The properties above will play a central role in the next sections.

Thanks to Hypothesis \ref{ipo1}, one can prove the following equivalence.
\begin{Proposition}\label{norms} Assume Hypothesis \ref{ipo1}. Then for all $u\in H^i_{\frac{1}{a}}(0, 1)$  the norms $\norm{u}_{H^i_{\frac{1}{a}}(0, 1)}$, $\| u\|_i$ and
$
\|u\|_{i, \sim }$, $i=1,2$,
are equivalent.
\end{Proposition}
\begin{proof} 

By \cite[Proposition 2.6]{cfr1}, one has that there exists $C>0$ such that
\[
\int_0^1\frac{u^2}{a}dx \le C \int_0^1 (u')^2dx,
\]
for all  $u \in L^2_{\frac{1}{a}}(0,1) \cap H^1_0(0,1)$.
Thus the thesis follows immediately if $i=1$. 

Now, assume $i=2$. Proceeding as for $i=1$ and applying the classical Hardy's inequality to $z:= u'$ (observe that $z \in H^1_0(0,1)$), we have
\[
\int_0^1\frac{u^2}{a}dx \le C \int_0^1 (u')^2dx \le C \int_0^1 \frac{z^2}{x^2}dx \le 4  C\int_0^1 (z')^2dx= 4 C\int_0^1 (u'')^2dx
\]
and the thesis follows.
\end{proof}

Hence, assuming Hypothesis \ref{ipo1} in the rest of the paper, we can use indifferently $\|\cdot\|_i$ or $\|u\|_{i, \sim }$ in place of $\|\cdot\|_{H^i_{\frac{1}{a}}(0, 1)}$, $i=1,2$.

Using the previous spaces, it is possible to define the operator $(A,D(A))$ by 
\begin{equation*}
	Au:=au'''' \quad \text{for all } u \in
	D(A):=\left\{u\in H^2_{\frac{1}{a}}(0, 1): au''''\in L^2_{\frac{1}{a}}(0, 1) \right\}.
\end{equation*}

Moreover, 
\begin{equation*}
	\langle Au,v\rangle_{L^2_{\frac{1}{a}}(0,1)}=\int_{0}^{1}u''v''dx,
\end{equation*}
i.e.
\begin{equation}\label{GF1}
	\int_{0}^{1}u''''v\,dx=\int_{0}^{1}u''v''dx,
\end{equation}
for all $(u,v)\in D(A) \times H^2_{\frac{1}{a}}(0,1)$ (see \cite[Proposition 2.1]{CF}).

Another important Hilbert space, related to the well posedness of \eqref{(P_1)}, is the following one
\begin{equation*}
	\mathcal{H}_0:=H^2_{\frac{1}{a}}(0,1)\times L^2_{\frac{1}{a}}(0,1),
\end{equation*}
endowed with the inner product
\begin{equation*}
	\langle (u,v),(\tilde{u},\tilde{v})\rangle_{\mathcal{H}_0}:=\int_{0}^{1}u''\tilde{u}''dx+\int_{0}^{1}\frac{v\tilde{v}}{a}dx
\end{equation*}
 and with the norm
\begin{equation*}
	\|(u,v)\|^2_{\mathcal{H}_0}:=\int_{0}^{1}(u'')^2dx+\int_{0}^{1}\frac{v^2}{a}dx
\end{equation*}
for every $(u,v), \;(\tilde{u},\tilde{v})\in\mathcal{H}_0$. Then, consider the matrix operator $\mathcal{A}:D(\mathcal{A})\subset\mathcal{H}_0\to \mathcal{H}_0$ given by
\[	\mathcal{A}:=\begin{pmatrix}
	0 & Id \\
	-A & 0
\end{pmatrix},\quad\,\,\,\,\,\,\,D(\mathcal{A}):=D(A)\times H^2_{\frac{1}{a}}(0,1).\]
Using this operator, we rewrite (\ref{(P_1)}) as a Cauchy problem. Indeed, setting
\[	\mathcal{U}(t):=\begin{pmatrix}
	y(t) \\
	y_t(t)
\end{pmatrix}\quad\,\,\,\,\text{and}\quad\,\,\,\,\mathcal{U}_0:=\begin{pmatrix}
y^0_T \\
y^1_T
\end{pmatrix},
\]
one has that (\ref{(P_1)}) can be formulated as
\begin{equation}\label{Cauchy problem}
	\begin{cases}
		\dot{\mathcal{U}}(t)=\mathcal{A}\,\mathcal{U}(t), &t\ge 0,\\
		\mathcal{U}(0)=\mathcal{U}_0.
	\end{cases}
\end{equation}
\begin{Theorem}\label{generator}
Assume Hypothesis $\ref{ipo1}$. Then the operator $(\cA, D(\cA))$ is non positive with dense domain and generates a contraction semigroup  $(S(t))_{t \ge 0}$. 
\end{Theorem}
\begin{proof}
According to \cite[Corollary 3.20]{nagel}, it is sufficient to prove that $ \mathcal A:D(\mathcal A)\to \mathcal H_0$ is dissipative and that $\mathcal I-\mathcal A$ is surjective, where
\[
\mathcal {I}:=\begin{pmatrix}
	Id & 0 \\
	0 & Id
\end{pmatrix}.
\]

\underline{$ \mathcal A$ is dissipative:} take $(u,v) \in D(\mathcal A)$. Then $(u,v) \in D(A)\times H^2_{\frac{1}{a}}(0,1)$ and so \eqref{GF1} holds. Hence, 
\[
\begin{aligned}
\langle \mathcal A (u,v), (u,v) \rangle_{\mathcal H_0} &=\langle (v, -Au), (u,v) \rangle _{\mathcal H_0}\\
&=\int_0^1 u''v''dx- \int_0^1 vAu\frac{1}{a}dx =0.
\end{aligned}
\]
By \cite[Chapter 2.3]{nagel}, the operator $\mathcal A$ is dissipative.

\underline{$\mathcal I - \mathcal A$ is surjective:} 
take  $(f,g) \in \mathcal H_0=H^2_{{\frac{1}{a}}}(0,1)\times L^2_{{\frac{1}{a}}}(0,1)$. We have to prove that there exists $(u,v) \in D(\mathcal A)$ such that
\begin{equation}\label{4.3'}
 ( \mathcal I-\mathcal A)\begin{pmatrix} u\\
v\end{pmatrix} = \begin{pmatrix}f\\
g \end{pmatrix} \Longleftrightarrow  \begin{cases} v= u -f,\\
Au + u= f+ g.\end{cases}
\end{equation}
Thus, define $F: H^2_{{\frac{1}{a}}}(0,1) \rightarrow \R$ as
\[
F(z)=\int_0^1(f+g) z\frac{1}{a}  dx,
\]
for all $z \in H^2_{{\frac{1}{a}}}(0,1)$.
Obviously, $F\in \left(H^{2}_{{\frac{1}{a}}}(0,1)\right)^*$, being $\left(H^{2}_{{\frac{1}{a}}}(0,1)\right)^*$ the dual space of $H^2_{{\frac{1}{a}}}(0,1)$ with respect to the pivot space $L^2_{{\frac{1}{a}}}(0,1)$. Now, introduce the bilinear form $L:H^{2}_{{\frac{1}{a}}}(0,1)\times H^{2}_{{\frac{1}{a}}}(0,1)\to \R$ given by
\[
L(u,z):=  \int_0^1 u z \frac{1}{a} dx + \int_0^1u''z''dx 
\]
for all $u, z \in H^{2}_{{\frac{1}{a}}}(0,1)$. Clearly, thanks to the equivalence of the norms given before, $L(u,z)$ is coercive. Moreover $L(u,z)$ is 
 continuous: indeed, for all $u,z \in H^{2}_{{\frac{1}{a}}}(0,1)$, we have
\[
|L(u,z)| \le  \|u\|_{ L^2_{\frac{1}{a}} (0,1) }\|z\|_ {L^2_{\frac{1}{a}} (0,1) } +\|u''\|_{L^2(0,1)}\|z''\|_ {L^2(0,1)}
\]
and the conclusion follows again by the equivalence of the norms.

As a consequence, by the Lax-Milgram Theorem, there exists a unique solution $u \in H^{2}_{{\frac{1}{a}}}(0,1)$ of
\[
L(u,z)= F(z)  \mbox{ for all }z\in H^{2}_{{\frac{1}{a}}}(0,1),\]
namely
\begin{equation}\label{4.4'}
\int_0^1 u z \frac{1}{a} dx + \int_0^1 u''z''dx = \int_0^1(f+g) z \frac{1}{a} dx 
\end{equation}
for all $z \in H^{2}_{{\frac{1}{a}}}(0,1)$.

Now, take $v:= u-f$; then $v \in H^{2}_{{\frac{1}{a}}}(0,1)$.
We will prove that $(u,v) \in D(\mathcal A)$ and solves \eqref{4.3'}. To begin with, \eqref{4.4'} holds for every $z \in \mathcal C_c^\infty(0,1).$ Thus we have
\[
\int_0^1 u''z''dx = \int_0^1(f+g-u) z \frac{1}{a} dx 
\]
 for every $z \in\mathcal  C_c^\infty(0,1).$ Hence $\ds (u'')''= (f+g- u)  \frac{1}{a}$ a.e. in $(0,1)$, i.e. $
Au = f+g-u
$ a.e. in $(0,1)$. Thus, as in \cite[Theorem 2.1]{CF}, $u \in D(A)$;  hence $(u,v) \in D(\mathcal A)$ and $ u +A u =f+g$.  
Recalling that $v = u - f$,  we have that $(u,v)$ solves \eqref{4.3'}.

\end{proof}

Now, if $\cU_0  \in \mathcal H_0$ then $\cU(t)= S(t)\cU_0$ is the mild solution of \eqref{Cauchy problem}. Also, if $\cU_0 \in D(\mathcal A)$, then the solution is classical and the equation in \eqref{(P_1)} holds for all $t \ge0$. Hence, by \cite[Propositions 3.1 and 3.3]{daprato}, one has the following theorem. 
\begin{Theorem}\label{Theorem 2.6}
	Assume Hypothesis \ref{ipo1}.
	If $(y^0_T,y^1_T)\in\mathcal{H}_0$, then there exists a unique mild solution
	\begin{equation*}
		y\in \mathcal{C}^1([0,+\infty);L^2_{\frac{1}{a}}(0,1))\cap \mathcal{C}([0,+\infty);H^2_{\frac{1}{a}}(0,1))
	\end{equation*}
of (\ref{(P_1)}) which depends continuously on the initial data $(y^0_T,y^1_T)\in \mathcal{H}_0$. Moreover, if $(y^0_T,y^1_T)\in D(\mathcal{A})$, then the solution $y$ is classical, in the sense that
\begin{equation*}
	y\in \mathcal{C}^2([0,+\infty);L^2_{\frac{1}{a}}(0,1))\cap \mathcal{C}^1([0,+\infty);H^2_{\frac{1}{a}}(0,1))\cap \mathcal{C}([0,+\infty);D(A))
\end{equation*}
and the equation of (\ref{(P_1)}) holds for all $t\ge 0$.
\end{Theorem}
\begin{Remark}\label{remark revers}
\begin{enumerate}
\item	Due to the reversibility (in time) of the  equation, solutions exist with the same regularity also for $t<0$.
\item Observe that the proofs of Theorems \ref{generator} and \ref{Theorem 2.6} are independent of \eqref{crescente}.
\end{enumerate} 
\end{Remark}

\section{Energy estimates}\label{Section 3}
In this section we prove some estimates of the energy associated to the solution of \eqref{(P_1)}. To this aim we give the next definition.

\begin{Definition}
Let $y$ be a mild solution of (\ref{(P_1)}) and consider its energy given by the continuous function defined as
\begin{equation*}
	E_y(t):=\frac{1}{2}\int_0^1 \Biggl (\frac{y^2_t(t,x)}{a(x)}+y^2_{xx}(t,x) \Biggr )dx\quad\,\,\,\,\,\,\forall\;t\ge 0.
\end{equation*}
\end{Definition}
The definition above guarantees that the classical conservation of the energy still holds true also in this degenerate situation.

\begin{Theorem}\label{teorema energia costante}
	Assume Hypothesis \ref{ipo1} and let $y$ be a  mild solution of (\ref{(P_1)}). Then
	\begin{equation}\label{E costante nel tempo}
		E_y(t)=E_y(0)\quad\,\,\,\,\,\,\,\forall\;t\ge 0.
	\end{equation}
\end{Theorem}
\begin{proof}
First of all suppose that $y$ is a classical solution. Then multiplying the equation
	\begin{equation*}
		y_{tt}+Ay=0
	\end{equation*}
by $\frac{y_t}{a}$, integrating over $(0,1)$ and using the formula of integration by parts (\ref{GF1}), one has
\begin{equation*}
	\begin{aligned}
		0&=\frac{1}{2}\int_0^1 \Bigl (\frac{y^2_t}{a}\Bigr )_tdx+\int_0^1 y_{xxxx}\,y_tdx\\
		&=\frac{1}{2}\frac{d}{dt}\Biggl (\int_0^1\Bigl (\frac{y^2_t}{a}+y^2_{xx}\Bigr )dx\Biggr )=\frac{d}{dt}E_y(t).
	\end{aligned}
\end{equation*}
Consequently the energy $E_y$ associated to $y$ is constant.

If $y$ is a mild solution, we approximate the initial data with more regular ones, obtaining associated classical solutions for which (\ref{E costante nel tempo}) holds. Thanks to the usual estimates we can pass to the limit and obtain the thesis.
\end{proof}

In the next results we establish some inequalities for
the energy from above and from below; these inequalities will be used in the next section to establish the controllability result.
 First of all, we start with the following theorem, whose proof is based on the next lemma.

\begin{Lemma}\label{Lemma_bordo}
	Assume Hypothesis \ref{ipo1}. 
	\begin{enumerate}
	\item
	 If   $y \in H^1_{\frac{1}{a}}(0,1)$, then   $\ds
\lim_{x\rightarrow 0} \frac{x}{a}y^2(x)=0.$
\item If $y\in D(A)$, then
$
		y''\in W^{1,1}(0,1).
	$
	\end{enumerate}
\end{Lemma}

The previous results are proved in \cite[Lemma 3.2.5]{BFM wave eq} and \cite[Proposition 3.2]{CF_Neumann}, respectively, anyway we rewrite their proof in the Appendix
to make the paper self-contained.
\begin{Theorem}\label{Teorema 1}
	Assume  $a$  (WD)  or (SD) at $0$. If $y$ is a classical solution of (\ref{(P_1)}), then $y_{xx}(\cdot,1)\in L^2(0,T)$  for any $T>0$ and
	\begin{equation}\label{prima uguaglianza}
		\begin{aligned}
			\frac{1}{2}\int_0^T y^2_{xx}(t,1)dt&=\int_0^1\Bigl [y_t\frac{x^2}{a}y_x \Bigr ]^{t=T}_{t=0}dx+\frac{1}{2}\int_{Q_T}\frac{x}{a}y^2_t\,\Bigl (2-\frac{xa'}{a}\Bigr )dx\,dt\\
			&+3\int_{Q_T}xy^2_{xx}dx\,dt.
		\end{aligned}
	\end{equation}
\end{Theorem}
\begin{proof}
	Multiply the equation in (\ref{(P_1)}) by $\displaystyle\frac{x^2y_x}{a}$ and integrate over $Q_T=(0,T)\times (0,1)$. Integrating by parts we obtain
	\begin{equation}\label{passaggi}
		\begin{aligned}
			0&=\int_{Q_T}y_{tt}\frac{x^2y_x}{a}dx\,dt+\int_{Q_T}x^2y_xy_{xxxx}dx\,dt\\
			&=\int_0^1\Bigl [y_t\frac{x^2y_x}{a} \Bigr ]^{t=T}_{t=0}dx-\int_{Q_T}y_t\frac{x^2}{a}y_{xt}dx\,dt+\int_{Q_T}x^2y_xy_{xxxx}dx\,dt\\
			&=\int_0^1\Bigl [y_t\frac{x^2y_x}{a} \Bigr ]^{t=T}_{t=0}dx-\int_{Q_T}\frac{1}{2}\frac{x^2}{a}(y^2_t)_xdx\,dt+\int_{Q_T}x^2y_xy_{xxxx}dx\,dt\\
			&=\int_0^1\Bigl [y_t\frac{x^2y_x}{a} \Bigr ]^{t=T}_{t=0}dx-\frac{1}{2}\int_0^T\Bigl [\frac{x^2}{a}y^2_t \Bigr ]^{x=1}_{x=0}dt+\frac{1}{2}\int_{Q_T}\Bigl (\frac{x^2}{a} \Bigr )'y^2_tdx\,dt\\
			&+\int_{Q_T}x^2y_xy_{xxxx}dx\,dt.
		\end{aligned}
	\end{equation}
Now, $\ds \Bigl (\frac{x^2}{a} \Bigr )'=\frac{2xa-x^2a'}{a^2}=\frac{x}{a}\Bigl (2-\frac{xa'}{a} \Bigr)$. Hence, (\ref{passaggi}) reads
\begin{equation}\label{equazione principale}
\begin{aligned}
	\int_0^1\Bigl [y_t\frac{x^2}{a}y_x \Bigr ]^{t=T}_{t=0}dx&-\frac{1}{2}\int_0^T\Bigl [\frac{x^2}{a}y^2_t \Bigr ]^{x=1}_{x=0}dt+\frac{1}{2}\int_{Q_T}\frac{x}{a}y^2_t\Bigl (2-\frac{xa'}{a} \Bigr )dx\,dt\\
	&+\int_{Q_T}x^2y_xy_{xxxx}dx\,dt=0.
	\end{aligned}
\end{equation}
Furthermore,
 by the regularity of the solution, $y_t \in H^2_{\frac{1}{a}}(0,1)\subset H^1_{\frac{1}{a}}(0,1)$, thus
 \[
\lim_{x \rightarrow 0} \frac{x^2}{a(x)}y^2_t(t,x)=0,
\]
by  Lemma \ref{Lemma_bordo}; therefore, by the boundary conditions of $y$, one has $\ds\frac{1}{a(1)} y_t^2(t,1)=0$.
Now, consider the term $\int_{Q_T}x^2y_xy_{xxxx}dx\,dt$, which is well defined; indeed, using the fact that $\frac{x^2}{\sqrt a}$ is non decreasing, we have that there exists a positive constant $C$ such that
\begin{equation*}
\Biggl |x^2\frac{1}{\sqrt a}y_x\,\sqrt a\,y_{xxxx}\Biggr | \le C\bigl |y_x\bigr | \bigl |\sqrt{a}y_{xxxx}\bigr |.
\end{equation*}
By hypothesis, $y_x$ and $\sqrt{a}y_{xxxx}$ belong to $L^2(0,1)$; thus, by the H\"older's inequality, $x^2y_xy_{xxxx}\in L^1(0,1)$.

 Let $\delta >0$ and write
\begin{equation}\label{secondo integrale}
	\int_{Q_T}x^2y_xy_{xxxx}dx\,dt=\int_0^T\int_0^\delta x^2y_xy_{xxxx}dx\,dt+\int_0^T\int_\delta^1x^2y_xy_{xxxx}dx\,dt.
\end{equation}
Obviously, by the absolute continuity of the integral,
\begin{equation}\label{limite0}
	\lim_{\delta\to 0}\int_0^T\int_0^\delta x^2y_xy_{xxxx}dx\,dt= 0.
\end{equation}
Now we will estimate the second term in (\ref{secondo integrale}). By definition of $D(A)$, setting $I:= (\delta, 1]$, we have $y_{xxxx}\in L^2(I)$, thus $y\in H^4 (I)$ by \cite[Lemma 2.1]{CF}. Hence, we can integrate by parts
\begin{equation}\label{espressione con BT}
	\begin{aligned}
	&	\int_0^T\int_\delta^1x^2y_xy_{xxxx}dx\,dt=\int_0^T[x^2y_xy_{xxx}]^{x=1}_{x=\delta}dt-\int_0^T\int_\delta^1(x^2y_x)_xy_{xxx}dx\,dt\\
		&=\int_0^T[x^2y_xy_{xxx}]^{x=1}_{x=\delta}dt-\int_0^T\int_\delta^1(2xy_x+x^2y_{xx})y_{xxx}dx\,dt\\
		&=\int_0^T[x^2y_xy_{xxx}]^{x=1}_{x=\delta}dt-\int_0^T\int_\delta^1 2xy_xy_{xxx}dx\,dt-\frac{1}{2}\int_0^T\int_\delta^1x^2(y^2_{xx})_xdx\,dt\\
		&=\int_0^T[x^2y_xy_{xxx}]^{x=1}_{x=\delta}dt-\int_0^T[2xy_xy_{xx}]^{x=1}_{x=\delta}dt+2\int_0^T\int_\delta^1(xy_x)_xy_{xx}dx\,dt\\
		&-\frac{1}{2}\int_0^T[x^2y^2_{xx}]^{x=1}_{x=\delta}dt+\frac{1}{2}\int_0^T\int_\delta^12xy^2_{xx}dx\,dt\\
		&=\int_0^T[x^2y_xy_{xxx}]^{x=1}_{x=\delta}dt-\int_0^T[2xy_xy_{xx}]^{x=1}_{x=\delta}dt-\frac{1}{2}\int_0^T[x^2y^2_{xx}]^{x=1}_{x=\delta}dt\\
		&+2\int_0^T\int_\delta^1y_xy_{xx}dx\,dt+2\int_0^T\int_\delta^1xy^2_{xx}dx\,dt+\int_0^T\int_\delta^1xy^2_{xx}dx\,dt\\
		&=\int_0^T[x^2y_xy_{xxx}]^{x=1}_{x=\delta}dt-\int_0^T[2xy_xy_{xx}]^{x=1}_{x=\delta}dt-\frac{1}{2}\int_0^T[x^2y^2_{xx}]^{x=1}_{x=\delta}dt\\
		&+\int_0^T\int_\delta^1(y^2_x)_xdx\,dt+3\int_0^T\int_\delta^1xy^2_{xx}dx\,dt\\
		&=\int_0^T[x^2y_xy_{xxx}]^{x=1}_{x=\delta}dt-\int_0^T[2xy_xy_{xx}]^{x=1}_{x=\delta}dt-\frac{1}{2}\int_0^T[x^2y^2_{xx}]^{x=1}_{x=\delta}dt\\
		&+\int_0^T[y^2_x]^{x=1}_{x=\delta}dt+3\int_0^T\int_\delta^1xy^2_{xx}dx\,dt.
	\end{aligned}
\end{equation}
But $xy^2_{xx}\in L^1(0,1)$ and using the absolute continuity of the integral,
we obtain
\begin{equation*}
	\lim_{\delta\to 0}\int_0^T\int_\delta^1xy^2_{xx}dx\,dt=\int_0^T\int_0^1xy^2_{xx}dx\,dt.
\end{equation*}
Now we evaluate the boundary terms that appear in (\ref{espressione con BT}). To this aim observe that, thanks to the boundary conditions of $y$,
\begin{equation*}
	\begin{aligned}
		&[x^2y_xy_{xxx}]^{x=1}_{x=\delta}-[2xy_xy_{xx}]^{x=1}_{x=\delta}-\frac{1}{2}[x^2y^2_{xx}]^{x=1}_{x=\delta}+[y^2_x]^{x=1}_{x=\delta}\\
		&=-\delta^2y_x(t,\delta)y_{xxx}(t,\delta)+2\delta y_x(t,\delta)y_{xx}(t,\delta)-\frac{1}{2}y^2_{xx}(t,1)+\frac{1}{2}\delta^2y^2_{xx}(t,\delta)-y^2_x(t,\delta).
	\end{aligned}
\end{equation*}
Hence, we have to estimate the following quantities:
\begin{equation*}
	\delta^2y_x(t,\delta)y_{xxx}(t,\delta),
\end{equation*}
\begin{equation*}
	\delta y_x(t,\delta)y_{xx}(t,\delta),
\end{equation*}
\begin{equation*}
	\delta^2y^2_{xx}(t,\delta),
\end{equation*}
\begin{equation*}
	y^2_x(t,\delta)
\end{equation*}
as $\delta$ goes to $0$. Naturally, since $y\in H^2_0(0,1)$, $y_x$ is a continuous function. This implies that
\begin{equation}\label{limite}
	\lim_{\delta\to 0}y^2_x(t,\delta)= y_x^2(t,0)=0.
\end{equation}
Thanks to Lemma \ref{Lemma_bordo}, 
\begin{equation*}
	\lim_{\delta\to 0}\delta^2y^2_{xx}(t,\delta)=0=\lim_{\delta\to 0}	\delta y_x(t,\delta)y_{xx}(t,\delta).
\end{equation*}
It remains to prove that
\begin{equation}\label{quarto BT}
	\exists\lim_{\delta\to 0}	\delta^2y_x(t,\delta)y_{xxx}(t,\delta)=0.
\end{equation}
By \eqref{limite} it is sufficient to prove that $\displaystyle\exists\lim_{\delta\to 0}\delta y_{xxx}(t,\delta)\in\mathbb{R}$. To this aim, we rewrite $\delta y_{xxx}(t,\delta)$ as
\begin{equation}\label{Koraidon}
\begin{aligned}
	\delta y_{xxx}(t,\delta)&=y_{xxx}(t,1)-\int_\delta^1(xy_{xxx}(t,x))_xdx=y_{xxx}(t,1)-\int_\delta^1y_{xxx}(t,x)dx\\
	&-\int_\delta^1xy_{xxxx}(t,x)dx.
	\end{aligned}
\end{equation}
Note that $xy_{xxxx}(t, x)=\sqrt{a(x)}y_{xxxx}(t, x)\frac{x}{\sqrt{a(x)}}\in L^2(0,1)\subseteq L^1(0,1)$ (indeed $\sqrt{a}y_{xxxx}\in L^2(0,1)$ and $\ds \frac{x}{\sqrt{a(x)}}\in L^{\infty}(0,1)$, thanks to \eqref{crescente}). Hence, by the absolute continuity of the integral
$\displaystyle\lim_{\delta\to 0}\int_\delta^1 xy_{xxxx}(x)dx = \int_0^1 xy_{xxxx}(x)dx.$
On the other hand,
\begin{equation*}
\begin{aligned}
	\int_\delta^1y_{xxx}(t,x)dx&=\int_\delta^1\Biggl (y_{xxx}(t,1)-\int_x^1y_{xxxx}(t,s)ds \Biggr )dx\\
	&=(1-\delta)y_{xxx}(t,1)-\int_\delta^1\int_x^1y_{xxxx}(t,s)ds\,dx.
\end{aligned}
\end{equation*}
Now we estimate the last term in the previous equation
\begin{equation}\label{Miraidon}
\begin{aligned}
	\int_\delta^1\int_x^1y_{xxxx}(t,s)ds\,dx&=\int_\delta ^1 \int_\delta ^s y_{xxxx}(t,s)dxds=\int_\delta^1y_{xxxx}(t,s)(s-\delta)ds\\
	&=\int_\delta^1sy_{xxxx}(t,s)ds-\delta \int_\delta^1y_{xxxx}(t,s)ds.
\end{aligned}
\end{equation}
As before,  $\displaystyle\lim_{\delta\to 0}\int_\delta^1 sy_{xxxx}(t,s)ds = \int_0^1 sy_{xxxx}(t,s)ds$.
Moreover, as far as the second term in the last member of (\ref{Miraidon}) is concerned, we have
\begin{equation*}
	\begin{aligned}
		0&<\delta \int_\delta^1|y_{xxxx}(t,s)|ds=\delta \int_\delta^1\sqrt{a(s)}\frac{|y_{xxxx}(t,s)|}{\sqrt{a(s)}}ds\\
		&\le \delta \Biggl (\int_\delta^1\frac{1}{a(s)}ds\Biggr )^{\frac{1}{2}}\Vert\sqrt{a}y_{xxxx}\Vert_{L^2(0,1)}\\
		&=\delta^{1- \frac{K}{2}}\Biggl (\int_\delta^1\frac{\delta^K}{a(s)}ds \Biggr )^{\frac{1}{2}}\Vert\sqrt{a}y_{xxxx}\Vert_{L^2(0,1)}\\
		&\le\delta^{1- \frac{K}{2}} \Biggl (\int_\delta^1\frac{s^K}{a(s)}ds \Biggr )^{\frac{1}{2}}\Vert\sqrt{a}y_{xxxx}\Vert_{L^2(0,1)}\\
		&\le C\delta^{1- \frac{K}{2}}(1-\delta)^{\frac{1}{2}}\Vert\sqrt{a}y_{xxxx}\Vert_{L^2(0,1)},	\end{aligned}
\end{equation*} for a positive constant $C$. Thus, since $K<2$, it follows that $\displaystyle\lim_{\delta\to 0}\delta \int_\delta^1y_{xxxx}(t,s)ds= 0$. Consequently, 
\begin{equation*}
		\lim_{\delta\to 0}	\int_\delta^1\int_x^1y_{xxxx}(t,s)ds\,dx = 	\lim_{\delta\to 0}	\int_\delta^1y_{xxxx}(t,s)(s-\delta)ds= 	\int_0^1sy_{xxxx}(t,s)ds
\end{equation*}
As a consequence, coming back to (\ref{Koraidon}),
\begin{equation*}
	\exists\lim_{\delta\to 0}\delta y_{xxx}(t,\delta) \in \R
\end{equation*}
and, in particular, (\ref{quarto BT}) is proved. 

Thus, by \eqref{espressione con BT}, one has
\[
\lim_{\delta\to 0}	\int_0^T\int_\delta^1x^2y_xy_{xxxx}dx\,dt= - \frac{1}{2}\int_0^T y_{xx}^2(t,1)dt + 3 \int_0^T\int_0^1 x y_{xx}^2dxdt.
\]
By the previous equality, \eqref{equazione principale}, \eqref{secondo integrale} and \eqref{limite0}, the thesis follows.

\end{proof}
As a consequence of the previous equality on $\frac{1}{2}\int_0^T y^2_{xx}(t,1)dt$, we have the next estimate from below on the energy.
\begin{Theorem}\label{Teorema 1bis}
	Assume  $a$  (WD)  or (SD) at $0$. If $y$ is a mild solution of \eqref{(P_1)}, then
\begin{equation}\label{stima E dall'alto}
	\int_0^Ty^2_{xx}(t,1)dt\le \Biggl (12T+4\max \Biggl \{\frac{
	4}{a(1)},1 \Biggr \} \Biggr )E_y(0).
\end{equation}
\end{Theorem}
\begin{proof}
As a first step, assume that $y$ is a classical solution of \eqref{(P_1)}; thus \eqref{prima uguaglianza} holds. Now, set $z(t,x):=y_x(t,x)$; since $y_x(t,0)=0$, by the classical Hardy's inequality we obtain
\begin{equation}\label{hardy}
	\int_0^1 y_x^2dx = \int_0^1z^2dx=\int_0^1\frac{z^2}{x^2}x^2dx\le \int_0^1\frac{z^2}{x^2}dx\le 4\int_0^1z^2_xdx=4\int_0^1y^2_{xx}dx.
\end{equation}
Thus, applying (\ref{crescente}), one has
\begin{equation*}
	\begin{aligned}
		\Biggl |\int_0^1\frac{x^2y_x(\tau,x)y_t(\tau,x)}{a(x)}dx \Biggr |&\le\frac{1}{2}\int_0^1\frac{x^4}{a(x)}y^2_x(\tau,x)dx+\frac{1}{2}\int_0^1\frac{y^2_t(\tau,x)}{a(x)}dx\\
		&\le \frac{1}{2a(1)}\int_0^1y^2_x(\tau,x)dx+\frac{1}{2}\int_0^1\frac{y^2_t(\tau,x)}{a(x)}dx\\
		&\le \frac{2}{a(1)}\int_0^1y^2_{xx}(\tau,x)dx+\frac{1}{2}\int_0^1\frac{y^2_t(\tau,x)}{a(x)}dx
	\end{aligned}
\end{equation*}
for all $\tau\in [0,T]$. By Theorem \ref{teorema energia costante}, we get
\begin{equation}\label{primo}
	\begin{aligned}
		\left|\int_0^1\Bigl [\frac{x^2y_x(\tau,x)y_t(\tau,x)}{a(x)} \Bigr ]^{\tau =T}_{\tau =0}dx\right|&\le \frac{2}{a(1)}\int_0^1y^2_{xx}(T,x)dx+\frac{1}{2}\int_0^1\frac{y^2_t(T,x)}{a(x)}dx\\
		&+\frac{2}{a(1)}\int_0^1y^2_{xx}(0,x)dx+\frac{1}{2}\int_0^1\frac{y^2_t(0,x)}{a(x)}dx\\
		&\le \max\Biggl \{\frac{4}{a(1)},1 \Biggr \}(E_y(T)+E_y(0))\\
		&=2\max\Biggl \{\frac{4}{a(1)},1 \Biggr \}E_y(0).
	\end{aligned}
\end{equation}
Moreover, using the fact that $x|a'|\le Ka$, we find
\begin{equation}\label{secondo}
	\begin{aligned}
		\Biggl |\int_{Q_T}\frac{x}{a}y^2_t\,\Bigl (2-\frac{xa'}{a}\Bigr )dx\,dt \Biggr |&\le\int_{Q_T}\frac{x}{a}y^2_t\,(2+K)dx\,dt\\
		&\le(2+K)\int_{Q_T}\frac{y^2_t}{a}dx\,dt.
	\end{aligned}
\end{equation}
Clearly, 
\begin{equation}\label{terzo}
	\int_{Q_T}xy^2_{xx}dx\,dt\le \int_{Q_T}y^2_{xx}dx\,dt
\end{equation}
and from \eqref{prima uguaglianza}, (\ref{primo}), (\ref{secondo}), (\ref{terzo}), we get (\ref{stima E dall'alto}) if $y$ is a classical solution of \eqref{(P_1)}.
Now, let $y$ be the mild solution associated to the initial data $(y_0, y_1) \in \mathcal H_0$. Then, consider a sequence $\{(y_0^n, y_1^n)\}_{n \in  \N} \subset D(\mathcal A)$ that converges to $(y_0, y_1)$ and let $y^n$ be the classical solution of \eqref{(P_1)} associated to $(y_0^n, y_1^n)$. Clearly $y^n$ satisfies \eqref{stima E dall'alto}; then, we can pass to the limit and conclude.
\end{proof}

Now, we will prove an inequality on the energy from above. To this aim, we need on $\ds\int_0^Ty_{xx}^2(t,1)dt$ an equality different from \eqref{prima uguaglianza}.
\begin{Theorem}\label{TheoremO}
	Assume  $a$ (WD)  or (SD) at $0$. If $y$ is a classical solution of (\ref{(P_1)}), then $y_{xx}(\cdot,1)\in L^2(0,T)$ for any $T>0$ and
		\begin{equation}\label{seconda uguaglianza}
		\begin{aligned}
			\frac{1}{2}\int_0^T y^2_{xx}(t,1)dt&=\int_0^1\Bigl [\frac{xy_ty_x}{a} \Bigr ]^{t=T}_{t=0}dx+\frac{1}{2}\int_{Q_T}\frac{y^2_t}{a}\Bigl (1-\frac{xa'}{a}\Bigr )dx\,dt+\frac{3}{2}\int_{Q_T}y^2_{xx}dx\,dt.
		\end{aligned}
	\end{equation}
\end{Theorem}
\begin{proof}
	Multiplying the equation in (\ref{(P_1)}) by $\displaystyle\frac{xy_x}{a}$ and integrating over $Q_T$, we obtain
	\begin{equation}\label{uguaglianza 2}
		\begin{aligned}
			0&=\int_0^1\Bigl [\frac{xy_xy_t}{a} \Bigr ]^{t=T}_{t=0}dx-\int_{Q_T}\frac{1}{2}\frac{x}{a}(y^2_t)_xdx\,dt+\int_{Q_T}xy_xy_{xxxx}dx\,dt\\
			&=\int_0^1\Bigl [\frac{xy_xy_t}{a} \Bigr ]^{t=T}_{t=0}dx-\frac{1}{2}\int_0^T\Bigl [\frac{x}{a}y^2_t \Bigr ]^{x=1}_{x=0}dt+\frac{1}{2}\int_{Q_T}\Bigl (\frac{x}{a}\Bigr )'y^2_tdx\,dt\\
			&+\int_{Q_T}xy_xy_{xxxx}dx\,dt.
		\end{aligned}
	\end{equation}
Now, $\ds \Bigl (\frac{x}{a}\Bigr )'=\frac{a-xa'}{a^2}=\frac{1}{a}\Bigl (1-\frac{xa'}{a}\Bigr )$. Hence, (\ref{uguaglianza 2}) reads
\begin{equation}\label{Natale 3}
\begin{aligned}
	\int_0^1\Bigl [\frac{xy_xy_t}{a} \Bigr ]^{t=T}_{t=0}dx&-\frac{1}{2}\int_0^T\Bigl [\frac{x}{a}y^2_t \Bigr ]^{x=1}_{x=0}dt+\frac{1}{2}\int_{Q_T}\frac{y^2_t}{a}\Bigl (1-\frac{xa'}{a}\Bigr )dx\,dt\\
	&+\int_{Q_T}xy_xy_{xxxx}dx\,dt=0.
	\end{aligned}
\end{equation}
As before\begin{equation*}
	\lim_{x\to 0}\frac{x}{a(x)}y^2_t(t,x)=0
\end{equation*}
and $\displaystyle\frac{1}{a(1)}y^2_t(t,1)=0$, so that $\displaystyle\int_0^T\Bigl [\frac{x}{a}y^2_t \Bigr ]^{x=1}_{x=0}dt=0$. In addition, the term $\int_{Q_T}xy_xy_{xxxx}dx\,dt$ is well defined since $xy_xy_{xxxx}=\frac{x}{\sqrt{a}}y_x\,\sqrt{a}y_{xxxx} \in L^1(0,1)$. Thus, we take $\delta >0$ and, as in the proof of Theorem \ref{Teorema 1}, we rewrite
\begin{equation*}
	\int_{Q_T}xy_xy_{xxxx}dx\,dt=\int_0^T\int_0^\delta xy_xy_{xxxx}dx\,dt+\int_0^T\int_\delta^1 xy_xy_{xxxx}dx\,dt.
\end{equation*}
Since $\ds xy_xy_{xxxx}\in L^1(0,1)$, we have
$\ds
	\lim_{\delta\to 0}\int_0^T\int_0^\delta xy_xy_{xxxx}dx\,dt=0.
$
Moreover, integrating by parts the second term of the previous equality and thanks to the boundary conditions on $y$, we have
\begin{equation*}
	\begin{aligned}
&		\int_0^T\int_\delta^1 xy_xy_{xxxx}dx\,dt=\int_0^T\Bigl [xy_xy_{xxx}\Bigr ]^{x=1}_{x=\delta }dt-\int_0^T\int_\delta^1(xy_x)_xy_{xxx}dx\,dt\\
		&=-\int_0^T\delta y_x(t,\delta)y_{xxx}(t,\delta)dt-\int_0^T\int_\delta^1y_xy_{xxx}dx\,dt-\frac{1}{2}\int_0^T\int_\delta^1x(y^2_{xx})_xdx\,dt\\
		&=-\int_0^T\delta y_x(t,\delta)y_{xxx}(t,\delta)dt-\int_0^T[y_xy_{xx}]^{x=1}_{x=\delta }dt+\int_0^T\int_\delta^1y^2_{xx}dx\,dt\\
		&-\frac{1}{2}\int_0^T[xy^2_{xx}]^{x=1}_{x=\delta }dt+\frac{1}{2}\int_0^T\int_\delta^1y^2_{xx}dx\,dt\\
		&=-\int_0^T\delta y_x(t,\delta)y_{xxx}(t,\delta)dt
		+\int_0^Ty_x(t,\delta)y_{xx}(t,\delta)dt-\frac{1}{2}\int_0^Ty^2_{xx}(t,1)dt\\
		&+\frac{1}{2}\int_0^T\delta y^2_{xx}(t,\delta)dt+\frac{3}{2}\int_0^T\int_\delta^1y^2_{xx}dx\,dt.
	\end{aligned}
\end{equation*}
Proceeding as in the proof of Theorem \ref{Teorema 1}, one can prove that $\ds \exists \; \lim_{\delta \rightarrow 0} \delta y_{xxx}(t,\delta) \in \R$; hence
\begin{equation*}
	\lim_{\delta\to 0}\delta y_x(t,\delta)y_{xxx}(t,\delta)=0.
\end{equation*}
Moreover, by Lemma \ref{Lemma_bordo}, we get
\begin{equation*}
	\lim_{\delta\to 0}y_x(t,\delta)y_{xx}(t,\delta)=0,\,\,\,\,\,\,\lim_{\delta\to 0}\delta y^2_{xx}(t,\delta)=0.
\end{equation*}
Hence
\begin{equation*}
	\lim_{\delta\to 0}\int_0^T\int_\delta^1 xy_xy_{xxxx}dx\,dt=-\frac{1}{2}\int_0^Ty^2_{xx}(t,1)dt+\frac{3}{2}\int_{Q_T}y^2_{xx}dx\,dt.
\end{equation*}
Coming back to (\ref{Natale 3}), it follows that
\begin{equation*}
	\int_0^1\Bigl [\frac{xy_xy_t}{a} \Bigr ]^{t=T}_{t=0}dx+\frac{1}{2}\int_{Q_T}\frac{y^2_t}{a}\Bigl (1-\frac{xa'}{a}\Bigr )dx\,dt-\frac{1}{2}\int_0^Ty^2_{xx}(t,1)dt+\frac{3}{2}\int_{Q_T}y^2_{xx}dx\,dt=0
\end{equation*}
and (\ref{seconda uguaglianza}) holds.
\end{proof}

Thanks to \eqref{seconda uguaglianza} we can prove the following estimate on the energy from above.
\begin{Theorem}\label{Theorem OI}
	Assume  $a$ (WD)  or (SD) at $0$. If $y$ is a mild solution of (\ref{(P_1)}), then
	\begin{equation*}
		\int_0^T y^2_{xx}(t,1)dt\ge \Biggl (T(2-K)-4\max\Biggl \{1,\frac{4}{a(1)}, \frac{4K}{a(1)}\Biggr \} \Biggr )E_y(0)
	\end{equation*}
for any $T>0$.
\end{Theorem}
\begin{proof}
	Multiplying the equation in (\ref{(P_1)}) by $\displaystyle\frac{-Ky}{2a}$ and integrating over $Q_T$, we have
	\begin{equation*}
		\begin{aligned}
			0&=-\frac{K}{2}\int_{Q_T}\frac{y_{tt}y}{a}dx\,dt-\frac{K}{2}\int_{Q_T}yy_{xxxx}dx\,dt\\
			&=-\frac{K}{2}\int_0^1\Bigl [y\frac{y_t}{a} \Bigr ]^{t=T}_{t=0}dx+\frac{K}{2}\int_{Q_T}\frac{y^2_t}{a}dx\,dt-\frac{K}{2}\int_{Q_T}y^2_{xx}dx\,dt,
		\end{aligned}
	\end{equation*}
thanks to  (\ref{GF1}).
Summing the previous equality to (\ref{seconda uguaglianza}) multiplied by $2$ and using the degeneracy condition (\ref{sup}), we have
\begin{equation*}
	\begin{aligned}
		\int_0^Ty^2_{xx}(t,1)dt&=2\int_0^1\Bigl [\frac{xy_ty_x}{a} \Bigr ]^{t=T}_{t=0}dx+\int_{Q_T}\frac{y^2_t}{a}\Bigl (1-\frac{xa'}{a}\Bigr )dx\,dt+3\int_{Q_T}y^2_{xx}dx\,dt\\
		&-\frac{K}{2}\int_0^1\Bigl [y\frac{y_t}{a} \Bigr ]^{t=T}_{t=0}dx+\frac{K}{2}\int_{Q_T}\frac{y^2_t}{a}dx\,dt-\frac{K}{2}\int_{Q_T}y^2_{xx}dx\,dt\\
		&=2\int_0^1\Bigl [\frac{xy_ty_x}{a} \Bigr ]^{t=T}_{t=0}dx-\frac{K}{2}\int_0^1\Bigl [y\frac{y_t}{a} \Bigr ]^{t=T}_{t=0}dx\\
		&+\int_{Q_T}\frac{y^2_t}{a}\Bigl (1-\frac{xa'}{a}+\frac{K}{2} \Bigr )dx\,dt+\Bigl (3-\frac{K}{2}\Bigr )\int_{Q_T}y^2_{xx}dx\,dt\\
		&\ge 2\int_0^1\Bigl [\frac{xy_ty_x}{a} \Bigr ]^{t=T}_{t=0}dx-\frac{K}{2}\int_0^1\Bigl [y\frac{y_t}{a} \Bigr ]^{t=T}_{t=0}dx\\
		&+\left(1-\frac{K}{2}\right)\int_{Q_T}\frac{y^2_t}{a}dx\,dt+\Bigl (1-\frac{K}{2}\Bigr )\int_{Q_T}y^2_{xx}dx\,dt\\
		&= 2\int_0^1\Bigl [\frac{xy_ty_x}{a} \Bigr ]^{t=T}_{t=0}dx-\frac{K}{2}\int_0^1\Bigl [y\frac{y_t}{a} \Bigr ]^{t=T}_{t=0}dx+(2-K)TE_y(0).
	\end{aligned}
\end{equation*}
Now, we analyze the boundary terms that appear in the previous relation. By \eqref{primo}
\begin{equation*}
	\begin{aligned}
		2\Biggl |\int_0^1\Bigl [\frac{xy_x(\tau,x)y_t(\tau,x)}{a(x)} \Bigr ]^{\tau =T}_{\tau =0}dx\Biggr | &\le 4\max\Biggl \{\frac{4}{a(1)},1 \Biggr \}E_y(0).
	\end{aligned}
\end{equation*}
Furthermore, by \eqref{crescente}
\begin{equation*}
	\begin{aligned}
		\Biggl |\frac{y(\tau,x)y_t(\tau,x)}{a(x)} \Biggr | \le \frac{1}{2}\frac{y^2_t(\tau,x)}{a(x)}+\frac{1}{2a(1)}\frac{y^2(\tau,x)}{x^2}
	\end{aligned}
\end{equation*}
for all $\tau \in [0,T]$; in particular, by Theorem \ref{teorema energia costante} and \eqref{hardy}, we have
\begin{equation*}
	\begin{aligned}
		\Biggl |\int_0^1\frac{y(\tau,x)y_t(\tau,x)}{a(x)}dx \Biggr |&\le \frac{1}{2}\int_0^1\frac{y^2_t(\tau,x)}{a(x)}dx+\frac{1}{2a(1)}\int_0^1\frac{y^2(\tau,x)}{x^2}dx\\
		&\le \frac{1}{2}\int_0^1\frac{y^2_t(\tau,x)}{a(x)}dx+\frac{4}{2a(1)}\int_0^1y^2_x(\tau,x)dx\\
		&\le \frac{1}{2}\int_0^1\frac{y^2_t(\tau,x)}{a(x)}dx+\frac{16}{2a(1)}\int_0^1y^2_{xx}(\tau,x)dx\\
		&\le \max\Biggl \{1,\frac{16}{a(1)}\Biggr \}E_y(0),
	\end{aligned}
\end{equation*}
for all $\tau \in [0,T]$.
Hence, 
\begin{equation*}
	\begin{aligned}
		\frac{K}{2}\Biggl |\int_0^1\Bigl [y\frac{y_t}{a} \Bigr ]^{\tau =T}_{\tau =0}dx \Biggr | &\le K\max\Biggl \{1,\frac{16}{a(1)}\Biggr \}E_y(0)
	\end{aligned}
\end{equation*}
and the thesis follows if $y$ is a classical solution. If $y$ is a mild solution, then we can proceed as in Theorem \ref{Teorema 1bis}.
\end{proof}

\section{Boundary observability and null controllabi\-lity}\label{Section 4}

Inspired by \cite{ABCG}, we give the next definition.
\begin{Definition}
	Problem (\ref{(P_1)}) is said to be observable in time $T>0$ via the second derivative at $x=1$ if there exists a constant $C>0$ such that for any $(y^0_T,y^1_T)\in \mathcal{H}_0$ the classical solution $y$ of (\ref{(P_1)}) satisfies
	\begin{equation}\label{3.21}
		CE_y(0)\le \int_0^Ty^2_{xx}(t,1)dt.
	\end{equation}
Moreover, any constant satisfying (\ref{3.21}) is called observability constant for (\ref{(P_1)}) in time $T$.
\end{Definition}
Setting 
\begin{equation*}
	C_T:=\sup\bigl \{C>0: \text{$C$ satisfies (\ref{3.21})}\bigr \},
\end{equation*}
we have that problem (\ref{(P_1)}) is observable if and only if 
\begin{equation*}
	C_T=\inf_{(y^0_T,y^1_T)\neq (0,0)}\frac{\int_0^Ty^2_{xx}(t,1)dt}{E_y(0)}>0.
\end{equation*}
The inverse of $C_T$, i.e. $c_T:=\frac{1}{C_T}$, is called the {\it cost of observability} (or the {\it cost of control}) in time $T$.

Theorem \ref{Theorem OI} admits the following straightforward corollary.
\begin{Corollary}\label{cor 3.12}
	Assume $a$ is (WD)  or (SD) at $0$. If 
	\begin{equation*}\label{ipo 3.22}
		T>\frac{4}{2-K}\max\Biggl \{1,\frac{4}{a(1)}, \frac{4K}{a(1)}\Biggr \},
	\end{equation*}
then (\ref{(P_1)}) is observable in time $T$. Moreover
\begin{equation*}
	T(2-K)-4\max\Biggl \{1,\frac{4}{a(1)}, \frac{4K}{a(1)}\Biggr \}\le C_T.
\end{equation*}
\end{Corollary}
\vspace{1cm}

In the following we will study the problem of null controllability for (\ref{(P)}). As a first step, we give the definition of a solution for (\ref{(P)}) by {\it transposition}, which permits low regularity on the notion of solution itself. Precisely:
\begin{Definition}
	Let $f\in L^2_{loc}[0,+\infty)$ and $(u_0,u_1)\in L^2_{\frac{1}{a}}(0, 1)\times \left(H^{2}_{{\frac{1}{a}}}(0,1)\right)^*$. We say that $u$ is a solution by transposition of (\ref{(P)}) if
	\begin{equation*}
		u\in \mathcal{C}^1\left(([0,+\infty);\left(H^{2}_{{\frac{1}{a}}}(0,1)\right)^*\right)\cap \mathcal{C}([0,+\infty);L^2_{\frac{1}{a}}(0,1))
	\end{equation*}
and for all $T>0$
\begin{equation}\label{equazione transpoition}
\begin{aligned}
		\left\langle u_t(T),v^0_T\right\rangle_{\left(H^{2}_{{\frac{1}{a}}}(0,1)\right)^*,H^{2}_{\frac{1}{a}}(0, 1)}&-\int_0^1\frac{1}{a}u(T)v^1_T\,dx=\left\langle u_1,v(0)\right\rangle_{\left(H^{2}_{{\frac{1}{a}}}(0,1)\right)^*,H^{2}_{\frac{1}{a}}(0, 1)}\\
	&-\int_0^1\frac{1}{a}u_0v_t(0,x)\,dx-\int_0^Tf(t)v_{xx}(t,1)dt
\end{aligned}
\end{equation}
for all $(v^0_T,v^1_T)\in \mathcal H_0$, where $v$ solves the backward problem
\begin{equation}\label{backward problem}
	\begin{cases}
		v_{tt}(t,x)+a(x)v_{xxxx}(t,x)=0, &(t,x)\in (0,+\infty) \times (0,1),\\
		v(t,0)=0,\,\,v_x(t,0)=0,&t>0,\\
		v(t,1)=0,\,\, v_x(t,1)=0, &t>0,\\
		v(T,x)=v^0_T(x),\,\,v_t(T,x)=v^1_T(x),&x\in(0,1).
	\end{cases}
\end{equation}
\end{Definition}

Observe that, by Theorem \ref{Theorem 2.6}, there exists a unique mild solution of  \eqref{backward problem} in $[T, +\infty)$. Now, setting $y(t,x):=v(T-t,x)$, one has that $y$ satisfies (\ref{(P_1)}) with $y^0_T(x)=v^0_T(x)$ and $y^1_T(x)=-v^1_T(x)$. Hence, we can apply Theorem \ref{Theorem 2.6} to  (\ref{(P_1)}) obtaining that there exists a unique mild solution $y$ of  (\ref{(P_1)})  in $[0, +\infty)$. In particular, there exists a unique mild solution $v$ of \eqref{backward problem} in $[0, T]$. Thus, we can conclude that there exists a unique mild solution
\begin{equation*}
	v\in \mathcal{C}^1([0,+\infty);L^2_{\frac{1}{a}}(0,1))\cap \mathcal{C}([0,+\infty);H^2_{\frac{1}{a}}(0,1))
\end{equation*}
of \eqref{backward problem} in $[0, +\infty)$
which depends continuously on the initial data $V_T:= (v^0_T,v^1_T)\in\mathcal{H}_0$.

By Theorem \ref{teorema energia costante} the energy is preserved in our setting, as well, so that the method of transposition done in \cite{ABCG} continues to hold thanks to (\ref{stima E dall'alto}). Therefore, there exists a unique solution by transposition $u\in \mathcal{C}^1([0,+\infty);H^{-2}_{\frac{1}{a}}(0, 1))\cap \mathcal{C}([0,+\infty);L^2_{\frac{1}{a}}(0,1))$ of (\ref{(P)}), i.e. a solution of (\ref{equazione transpoition}). To prove this fact, consider the functional
$
\mathcal G:  \mathcal H_0 \rightarrow \R$ given by
\begin{equation}\label{G1}
\mathcal G(v^0_T, v^1_T) = \left\langle u_1,v(0)\right\rangle_{\left(H^{2}_{{\frac{1}{a}}}(0,1)\right)^*,H^{2}_{\frac{1}{a}}(0, 1)}-\int_0^1\frac{1}{a}u_0v_t(0,x)\,dx-\int_0^Tf(t)v_{xx}(t,1)dt,
\end{equation}
for all $T>0$,
where $v$ solves \eqref{backward problem}.
Clearly, $\mathcal G$ is linear. Moreover, it is continuous. Indeed, for all $T>0$, we have
\[
\begin{aligned}
\begin{aligned}
|\mathcal G(v^0_T, v^1_T)|&\le  \|u_1\|_{\left(H^{2}_{{\frac{1}{a}}}(0,1)\right)^*}\|v(0)\|_{H^{2}_{\frac{1}{a}}(0, 1)}+ \|u_0\|_{L^2_{\frac{1}{a}}(0, 1)}\|v_t(0)\|_{L^{2}_{\frac{1}{a}}(0, 1)}\\
&+\|f\|_{L^2(0,T)}\left(\int_0^Tv_{xx}^2(t,1)dt\right)^{\frac{1}{2}}.
\end{aligned}
\end{aligned}
\]
By \eqref{stima E dall'alto} and Theorem \ref{teorema energia costante}, there exists a positive constant $C$ such that
\[\begin{aligned}
\int_0^Tv_{xx}^2(t,1)dt \le C E_v(0)&=CE_v(T)=\frac{C}{2}\int_0^1 \Biggl (\frac{v^2_t(T,x)}{a(x)}+v^2_{xx}(T,x) \Biggr )dx\\
&=\frac{C}{2}\|(v_T^0, v^1_T)\|_{\mathcal H_0}.
\end{aligned}
\]
Hence
\[\begin{aligned}
|\mathcal G(v^0_T, v^1_T)| &\le \|u_1\|_{\left(H^{2}_{{\frac{1}{a}}}(0,1)\right)^*}\|v(0)\|_{H^{2}_{\frac{1}{a}}(0, 1)}+ \|u_0\|_{L^2_{\frac{1}{a}}(0, 1)}\|v_t(0)\|_{L^{2}_{\frac{1}{a}}(0, 1)}\\
&+\|f\|_{L^2(0,T)}\|(v_T^0, v^1_T)\|_{\mathcal H_0}.
\end{aligned}\]
Using again Theorem \ref{teorema energia costante}, we have $\|v_t(0)\|_{L^{2}_{\frac{1}{a}}(0, 1)}\le E_v(T)$; thus $\|v_t(0)\|_{L^{2}_{\frac{1}{a}}(0, 1)}\le C\|(v_T^0, v^1_T)\|_{\mathcal H_0}$.
Analogously, thanks to Proposition \ref{norms}, there exists $C>0$ such that
$
\|v(0)\|_{H^{2}_{\frac{1}{a}}(0, 1)} \le  C\|(v_T^0, v^1_T)\|_{\mathcal H_0}.
$
Thus we can conclude that there exists $C>0$ so that
\[
|\mathcal G(v^0_T, v^1_T) | \le  C\|(v_T^0, v^1_T)\|_{\mathcal H_0},
\]
i.e.  $\mathcal G$ is continuous.

Being $\mathcal G \in (\mathcal H_0)^*= \left(H^{2}_{{\frac{1}{a}}}(0,1)\right)^* \times L^2_{\frac{1}{a}}(0,1),$ we can use the Riesz Theorem obtaining that for any $T>0$ there exists a unique $(\tilde u_{T}^0, \tilde u_{T}^1) \in  (\mathcal H_0)^*$ such that
\begin{equation}\label{G2}
\begin{aligned}
\mathcal G(v^0_T, v^1_T)&= \langle (\tilde u_T^0, \tilde u_T^1) , (v^0_T, v^1_T)\rangle_{(\mathcal H_0)^*, \mathcal H_0}\\
& = \langle \tilde u_T^0, v^0_T\rangle_{\left(H^{2}_{{\frac{1}{a}}}(0,1)\right)^*, H^{2}_{\frac{1}{a}}(0, 1)} + \int_0^1 \frac{\tilde u_T^1 v^1_T}{a}dx.
\end{aligned}
\end{equation}
Moreover, $\tilde u_{T}^0, \tilde u_{T}^1$ depend continuously on $T$, so there exists a unique $u\in \mathcal C([0, +\infty); L^2_{\frac{1}{a}}(0,1)) \cap \mathcal C^1([0, +\infty); H^{-2}_{\frac{1}{a}} (0,1))$ such that $u(T)= -\tilde u_{T}^1$ and $u_t(T)= \tilde u_{T}^0$. By \eqref{G1} and \eqref{G2}, we can conclude that $u$ is the unique solution by transposition of \eqref{(P)}.
\vspace{1cm}

Now, we are ready to examine null controllability. To this aim, consider the bilinear form $\Lambda:\mathcal{H}_0\times\mathcal{H}_0\to\mathbb{R}$ defined as
	\begin{equation*}
		\Lambda(V_T,W_T):=\int_0^Tv_{xx}(t,1)w_{xx}(t,1)dt,
	\end{equation*}
where $v$ and $w$ are the solutions of (\ref{backward problem}) associated to the data $V_T:= (v^0_T,v^1_T)$ and $W_T:= (w^0_T,w^1_T)$, respectively. The following lemma holds.
\begin{Lemma}\label{lemma 4.2}
	Assume $a$  (WD)  or (SD) at $0$. The bilinear form $\Lambda$ is continuous and coercive.
\end{Lemma}
\begin{proof}
	By Theorem \ref{teorema energia costante} $E_v$ and $E_w$ are constant in time and, due to (\ref{stima E dall'alto}), one has that $\Lambda$ is continuous. Indeed, by Holder's inequality and \eqref{stima E dall'alto},
	\begin{equation*}
		\begin{aligned}
			|\Lambda(V_T,W_T)|&\le \int_0^T|v_{xx}(t,1)w_{xx}(t,1)|\,dt\\
			&\le \Biggl (\int_0^Tv^2_{xx}(t,1)dt\Biggr )^{\frac{1}{2}}\Biggl (\int_0^Tw^2_{xx}(t,1)dt\Biggr )^{\frac{1}{2}}\\
			&\le CE^{\frac{1}{2}}_v(T)E^{\frac{1}{2}}_w(T)\\
			&=C \left( \int_0^1 \frac{(v^1_T)^2(x)}{a}dx +\int_0^1 [(v^0_T)_{xx}]^2(x) dx\right)^{\frac{1}{2}}\cdot\\
&\left( \int_0^1 \frac{(w^1_T)^2(x)}{a}dx + \int_0^1 [(w^0_T)_{xx}]^2(x)dx\right)^{\frac{1}{2}}\\
			&= C\Vert (v(T),v_t(T))\Vert_{\mathcal{H}_0}\Vert (w(T),w_t(T))\Vert_{\mathcal{H}_0}=C\Vert V_T\Vert_{\mathcal{H}_0}\Vert W_T\Vert_{\mathcal{H}_0}
		\end{aligned}
	\end{equation*}
for a positive constant $C$ independent of $(V_T,W_T)\in \mathcal{H}_0\times\mathcal{H}_0$.

In a similar way one can prove that $\Lambda$ is coercive. Indeed, by Theorems \ref{Theorem OI}, one immediately has that there exists $C>0$ such that
\begin{equation*}
	\Lambda(V_T,V_T)=\int_0^Tv^2_{xx}(t,1)dt\ge C_TE_v(0)=C_TE_v(T)\ge C\Vert V_T\Vert_{\mathcal{H}_0}^2
\end{equation*}
for all $V_T\in\mathcal{H}_0$.
\end{proof}
Function $\Lambda$ is used to prove the null controllability property for the original problem (\ref{(P)}). To this aim, let us start defining $T_0$ as the lower bound found in Corollary \ref{cor 3.12}, i.e.
\[
T_0:=\frac{4}{2-K}\max\Biggl \{1,\frac{4}{a(1)}, \frac{4K}{a(1)}\Biggr \}.
\]
\begin{Theorem}
	Assume $a$ (WD)  or (SD) at $0$. Then, for all $T>T_0$ and for every $(u_0,u_1)\in L^2_{\frac{1}{a}}(0, 1)\times H^{-2}_{\frac{1}{a}}(0, 1)$, there exists a control $f\in L^2(0,T)$ such that the solution of (\ref{(P)}) satisfies
	\begin{equation}\label{tesi finale}
			u(T,x)=u_t(T,x)=0\,\,\,\,\,\,\forall\; x\in (0,1).
	\end{equation}
\end{Theorem}
\begin{proof}
	Consider the map $\mathcal{L}:\mathcal{H}_0\to\mathbb{R}$ given by
	\begin{equation*}
		\mathcal{L}(V_T):=\left\langle u_1,v(0)\right\rangle_{\left(H^{2}_{{\frac{1}{a}}}(0,1)\right)^*,H^{2}_{\frac{1}{a}}(0, 1)}-\int_0^1\frac{u_0v_t(0,x)}{a}dx,
	\end{equation*}
where $v$ is the solution of (\ref{backward problem}) associated to the initial data $V_T:=(v^0_T,v^1_T)\in\mathcal{H}_0$. Clearly, $\mathcal{L}$ is continuous  and linear and, thanks to Lemma \ref{lemma 4.2}, we can apply the Lax-Milgram Theorem. Thus, there exists a unique $\overline{V}_T\in\mathcal{H}_0$ such that
\begin{equation*}
	\Lambda(\overline{V}_T,W_T)=\mathcal{L}(W_T)\,\,\,\,\,\,\,\,\forall\;W_T\in\mathcal{H}_0.
\end{equation*}
Set $f(t):=\overline{v}_{xx}(t,1)$, where $\overline v$ is the unique solution of (\ref{backward problem}) associated to $\overline{V}_T$. Then 
\begin{equation}\label{4.4}
	\begin{aligned}
		\int_0^Tf(t)w_{xx}(t,1)dt&=\int_0^T\overline v_{xx}(t,1)w_{xx}(t,1)dt=\Lambda(\overline V_T,W_T)=\mathcal{L}(W_T)\\
		&=\left\langle u_1, w(0)\right\rangle_{\left(H^{2}_{{\frac{1}{a}}}(0,1)\right)^*,H^{2}_{\frac{1}{a}}(0, 1)}-\int_0^1\frac{u_0w_t(0,x)}{a}dx
	\end{aligned}
\end{equation}
for all $W_T\in\mathcal{H}_0$. 

Finally, denote by $u$ the solution by transposition of (\ref{(P)}) associated to the function $f$ introduced above. We have that
\begin{equation}\label{4.5}
	\begin{aligned}
		\int_0^Tf(t)w_{xx}(t,1)dt&=-\left\langle u_t(T),w^0_T\right\rangle_{\left(H^{2}_{{\frac{1}{a}}}(0,1)\right)^*,H^{2}_{\frac{1}{a}}(0, 1)}+\int_0^1\frac{u(T)w^1_T}{a}dx\\
		&+\left\langle u_1,w(0)\right\rangle_{\left(H^{2}_{{\frac{1}{a}}}(0,1)\right)^*,H^{2}_{\frac{1}{a}}(0, 1)}-\int_0^1\frac{u_0w_t(0,x)}{a}dx.
	\end{aligned}
\end{equation}
Combining (\ref{4.4}) and (\ref{4.5}), it follows that
\begin{equation*}
	\left\langle u_t(T),w^0_T\right\rangle_{\left(H^{2}_{{\frac{1}{a}}}(0,1)\right)^*,H^{2}_{\frac{1}{a}}(0, 1)}-\int_0^1\frac{u(T)w^1_T}{a}dx=0
\end{equation*}
for all $(w^0_T,w^1_T)\in\mathcal{H}_0$. Hence, we have (\ref{tesi finale}).
\end{proof}

	\section{Appendix}\label{Appendix}
	\begin{proof}[Proof of Lemma \ref{Lemma_bordo}.1]
If $K<1$, where $K$ is the constant of Hypothesis \ref{ipo1}, then the assertion follows immediately by \eqref{ipotesi limite} with $\gamma=1$. Thus assume $K\ge 1$.
Set $\ds z(x):= \frac{x}{a(x)}y^2(x)$. Then $z \in L^1(0,1)$. Indeed
\[
\int_0^1\frac{x}{a} u^2dx \le  \int_0^1 \frac{u^2}{a}dx.
\]
Moreover
$
\ds \ds z'= \frac{u^2}{a}+ 2\frac{xuu'}{a} - \frac{a'x}{a^2}u^2;
$
thus, for a suitable $\ve>0$ given by Hypothesis \ref{ipo1},
\[
\begin{aligned}
\int_0^\ve|z'|dx &\le \int_0^1 \frac{u^2}{a}dx + 2 \left(\int_0^\ve \frac{x^2(u')^2}{a}dx\right)^{\frac{1}{2}}\left( \int_0^1 \frac{u^2}{a}dx\right)^{\frac{1}{2}} + K\int_0^1 \frac{u^2}{a}dx\\
&\le (1+K)\int_0^1 \frac{u^2}{a}dx + \frac{2\ve}{\sqrt{a(\ve)}}\left(\int_0^1 (u')^2dx\right)^{\frac{1}{2}}\left( \int_0^1 \frac{u^2}{\sigma}dx\right)^{\frac{1}{2}}.
\end{aligned}
\]
This is enough to conclude that  $z \in W^{1,1}(0,1)$ and thus there exists  $\ds \lim_{x\to 0}z(x)=L\in\R$. If $L\neq0$, sufficiently close to $x=0$ we would have that
$
\ds \frac{u^2(x)}{a}\geq \frac{|L|}{2x}\not\in L^1(0,1),
$
while $\ds\frac{u^2}{a} \in L^1(0,1)$.
\end{proof}
	\begin{proof}[Proof of Lemma \ref{Lemma_bordo}.2]
In order to prove the lemma, fixed $y\in D(A)$, we will establish that $y''$ is absolutely continuous in $[0,1]$. To this aim,  let $\delta >0$. Clearly
	\begin{equation}\label{y secondo in delta*}
			y''(\delta)=y''(1)-\int_\delta^1y'''(x)dx,
			\end{equation}
			thus
			\begin{equation}\label{y secondo in delta}
		\begin{aligned}
			y''(\delta)&=y''(1)-\int_\delta^1\Biggl (y'''(1)-\int_x^1y''''(s)ds \Biggr )dx\\
			&=y''(1)-(1-\delta)y'''(1)+\int_\delta^1\Biggl (\int_x^1y''''(s)ds \Biggr )dx\\
			&=y''(1)-y'''(1)+\delta y'''(1)+\int_\delta^1\int_\delta^sy''''(s)dx\,ds\\
			&=y''(1)-y'''(1)+\delta y'''(1)+\int_\delta^1y''''(s)(s-\delta)ds.
		\end{aligned}
	\end{equation}
Trivially $\displaystyle\lim_{\delta\to 0}\delta y'''(1)= 0$ and, proceeding as in Theorem \ref{Teorema 1}, we have
\begin{equation*}\label{Miraidon 1}
			\lim_{\delta\to 0}	\int_\delta^1y''''(s)(s-\delta)ds= 	\int_0^1sy''''(t,s)ds.
\end{equation*}
Thus, if we pass to the limit as $\delta\to 0$ in (\ref{y secondo in delta}), we conclude that
\begin{equation*}
	\exists \lim_{\delta\to 0}y''(\delta)= y''(1)- y'''(1)+\int_0^1sy''''(s)ds \in\mathbb{R}.
\end{equation*}
By continuity, it is possible to define
$
	y''(0):=\lim_{\delta\to 0}y''(\delta).
$
In particular, by \eqref{y secondo in delta*},
\begin{equation*}
	\int_0^1y'''(x)dx=y''(1)-y''(0)
\end{equation*}
and the thesis follows.
\end{proof}

\end{document}